\documentclass[12pt,leqno]{article}
\usepackage{amsmath,amsthm,amssymb}
\usepackage{tabularx}
\newcommand{\GF}{{\mathbb F}}
\newcommand{\FF}{{\mathbb F}}

\newcommand{\ZZ}{{\mathbb Z}}
\newcommand{\R}{{\mathbb R}}

\newcommand{\wt}{{\rm wt}}

\DeclareMathOperator{\Harm}{Harm}

\usepackage{color} 

\newtheorem{Thm}{Theorem}[section]

\newtheorem{Lem}[Thm]{Lemma}

\newtheorem{Prop}[Thm]{Proposition}

\theoremstyle{definition}
\newtheorem{Def}[Thm]{Definition}

\newtheorem{Problem}[Thm]{Problem}

\makeatletter
\@addtoreset{equation}{section}

\makeatother

\begin{document}

\title{On the support $t$-designs of extremal Type III and IV codes}

\author{
Tsuyoshi Miezaki\thanks{
Faculty of Science and Engineering, 
Waseda University, 
Tokyo 169--8555, Japan, 
E-mail: miezaki@waseda.jp 
(Corresponding author)
}
and 
Hiroyuki Nakasora\thanks{
Institute for Promotion of Higher Education, Kobe Gakuin University, Kobe
651--2180, Japan,
E-mail: nakasora@ge.kobegakuin.ac.jp
}
}

\date{}

\maketitle

\begin{abstract}
 
Let $C$ be an extremal Type III or IV code
and $D_{w}$ be the support design of $C$ for weight $w$.
We introduce the numbers, $\delta(C)$ and $s(C)$, as follows: 
$\delta(C)$ is the largest integer $t$ such that, for all weights, $D_{w}$ is a $t$-design;  
$s(C)$ denotes the largest integer $t$ such that $w$ exists and $D_{w}$ is a $t$-design. 
Herein, we consider the possible values of $\delta(C)$ and $s(C)$.  

\end{abstract}




\setcounter{section}{+0}

\section{Introduction}

Let $D_{w}$ be the support design of code $C$ for weight $w$. 
From the Assmus--Mattson theorem \cite{assmus-mattson} 
, if $C$ is an extremal Type III (resp.~Type IV) code, 
then for all $w$, $D_{w}$ is a $5$-, $3$-, and $1$-design 
for $n=12m$ (resp.~$n=6m$), $12m+4$ (resp.~$n=6m+2$), and $12m+8$ (resp.~$n=6m+4$), 
respectively. 

Let 
\begin{align*}
\delta(C)&:=\max\{t\in \mathbb{N}\mid \forall w, 
D_{w} \mbox{ is a } t\mbox{-design}\},\\ 
s(C)&:=\max\{t\in \mathbb{N}\mid \exists w \mbox{ s.t.~} 
D_{w} \mbox{ is a } t\mbox{-design}\}.
\end{align*}
It is noteworthy that $\delta(C) \leq s(C)$. 
In our previous papers 
\cite{
{extremal design H-M-N},
{extremal design2 M-N},
{support design triply even code 48 M-N},
{Miezaki-Munemasa-Nakasora}}, 
we considered the following problems. 
\begin{Problem}\label{problem:1}
Find the upper bound of $s(C)$.
\end{Problem}
\begin{Problem}\label{problem:2}
Does the case where $\delta(C) < s(C)$ occur?
\end{Problem}

Next, we explain our motivation for this study. 
The first motivation is as follows. 
For Problem \ref{problem:1}, 
many examples of $5$-designs can be obtained from the Assmus--Mattson theorem;
however, an example of a $6$-design is not known. 
Therefore, we aim to obtain a $t$-design for $t \geq 6$ using the Assmus--Mattson theorem. 
For Problem \ref{problem:2}, if $C$ is an extremal Type II code, 
an example of $\delta(C)<s(C)$ \cite{extremal design2 M-N} does not exist. 
In \cite{support design triply even code 48 M-N}, 
we discovered the first nontrivial examples of $\delta(C)<s(C)$ 
in triply even binary codes of length $48$, 
an example of which is the moonshine code \cite{miyamoto}. 
Using this result, 
we provide a new characterization of the moonshine code \cite{miyamoto}.

The second motivation for this study is that 
the Assmus--Mattson theorem is one of the most 
important theorems in design and coding theory. 
Assmus--Mattson-type theorems 
in lattice and vertex operator algebra theories are 
known as the Venkov and H\"ohn theorems, respectively \cite{{Venkov},{H1}}. 
For example, the $E_8$-lattice and moonshine vertex operator algebra 
$V^{\natural}$ provide 
spherical $7$-designs for all $(E_8)_{2m}$ and 
conformal $11$-designs for all $(V^{\natural})_{m}$, $m>0$. 
It is noteworthy that 
the $(E_8)_{2m}$ and $(V^{\natural})_{m+1}$ 
are a spherical $8$-design and a conformal $12$-design, respectively, if and only if 
$\tau(m)=0$, where 
\[
q\prod_{m=1}^{\infty}(1-q^m)^{24}=\sum_{m=0}^{\infty}\tau(m)q^m 
\]
Furthermore, D.H.~Lehmer conjectured in \cite{Lehmer} that 
\[
\tau(m)\neq 0
\]
for all $m$ \cite{{Miezaki},{Venkov},{Venkov2}}. 
Therefore, it is interesting to determine the lattice $L$
(resp.~vertex operator algebra $V$) such that 
$L_m$ (resp.~$V_m$) are spherical (resp.~conformal) $t$-designs 
for all $m$ by the Venkov theorem (resp.~H\"ohn theorem) and 
$L_m'$ (resp.~$V_m'$) are spherical (resp.~conformal) $t'$-designs 
for some $m'$ with some $t'>t$.
This study is inspired by these possibilities. 
For related results, 
see
\cite{{Bannai-Koike-Shinohara-Tagami},{BM1},{BM2},{MN-St-2},
{BMY},
{extremal design H-M-N},{Miezaki2},{Miezaki-Munemasa-Nakasora},{extremal design2 M-N},{extremal design3 M-N}}. 

Next, we explain our main results. 
Herein, we present  
Problems \ref{problem:1} and \ref{problem:2} 
for extremal Type III and IV codes.
Let $C$ be an extremal Type III or IV code of length $n$.
In 1999, Zhang \cite{Zhang(1999)} showed that $C$ does not exist 
if 
\[
n=
\left\{
\begin{array}{l}
12m\ (m \geq 70), \\
12m+4\ (m \geq 75), \\
12m+8\ (m \geq 78),
\end{array}
\right.
\]
for Type III, and 
\[
n=
\left\{
\begin{array}{l}
6m\ (m \geq 17), \\
6m+2\ (m \geq 20), \\
6m+4\ (m \geq 22), 
\end{array}
\right.
\]
for Type IV. 
The proof of this fact is to show that the coefficient of 
$x^{n-d}y^d\ (d=\min(C))$ of the extremal weight enumerators is negative. 
In \cite{Zhang(1999)}, Zhang remarked that 
the bounds for Type III may be improved 
if one considers the coefficients 
of the highest and next-to-highest powers of $y$ 
of the extremal weight enumerators. 
We remark that using all the coefficients of the 
extremal weight enumerators, 
we obtain more strict bounds for Type III codes:
\begin{Thm}
Let $C$ be an extremal Type III code of length $n$.
Then $C$ does not exist 
if 
\[
n=
\left\{
\begin{array}{l}
12m\ (m\in \{6,8,10,12,14,16,18,20\}\cup \{m\in \ZZ\mid m\geq 22\}), \\
12m+4\ (m\in \{14,16,18,20,22,24,26,28\}\cup \{m\in \ZZ\mid m\geq 30\}), \\
12m+8\ (m\in \{22,24,26,28,30,32,34,36\}\cup \{m\in \ZZ\mid m\geq 38\}).
\end{array}
\right.
\]
\end{Thm}
This means that if there exists an extremal ternary code of 
length $n=12m+r\ (r\in \{0,4,8\})$, then 
$m$ must be in the following set: 
\begin{align}\label{boundIII}
m\in 
\left\{
\begin{array}{l}
\{i\in \ZZ\mid 1\leq i\leq 5\}\cup \{7,9,11,13,15,17,19,21\}\ \mbox{if }r= 0, \\
\{i\in \ZZ\mid 1\leq i\leq 13\}\cup \{15,17,19,21,23,25,27,29\}\ \mbox{if }r= 4, \\
\{i\in \ZZ\mid 1\leq i\leq 21\}\cup \{23,25,27,29,31,33,35,37\}\ \mbox{if }r= 8. 
\end{array}
\right.
\end{align}
The extremal weight enumerators for Type III codes of 
length $n\leq 12\times 78 +8=944)$ are listed in one of 
the author's homepage \cite{Miezakihomepage}. 

The main results of the present study are the following theorems. 

\begin{Thm}\label{thm:main upper bound III}
Let $C$ be an extremal Type III code of length $n$.

\begin{enumerate}
\item[$(1)$] 

Assume that $n=12m$. 
\begin{enumerate}
\item [{\rm (a)}]
If $m \neq 15$, $\delta(C)=s(C)=5$.
\item [{\rm (b)}]
If $m =15$, $\delta(C)=s(C)=5$ or $7$.
\end{enumerate}

\item[$(2)$] 



Assume that $n=12m+4$. 
\begin{enumerate}
\item [{\rm (a)}]
If $m \notin \{11,21,25\}$, $\delta(C)=s(C)=3$. 
\item [{\rm (b)}]
If $m \in \{11,21,25\}$, $\delta(C)=s(C)=3$ or $5$. 

\end{enumerate}

\item[$(3)$] 

Assume that $n=12m+8$. 
\begin{enumerate}
\item [{\rm (a)}]
If $m \neq 14$, $\delta(C)=s(C)=1$. 
\item [{\rm (b)}]
If $m = 14 $, $\delta(C)=s(C)=1$ or $3$. 

\end{enumerate}
\end{enumerate}

\end{Thm}

\begin{Thm}\label{thm:main upper bound IV}
Let $C$ be an extremal Type IV code of length $n$.

\begin{enumerate}
\item[$(1)$] 

Assume that $n=6m$ $(m \neq 1,2)$. 
\begin{enumerate}
\item [{\rm (a)}]
If $m \notin \{10,15\}$, $\delta(C)=s(C)=5$. 
\item [{\rm (b)}]
If $m \in \{10,15\}$, $\delta(C)=s(C)=5$ or $7$.
\end{enumerate}

\item[$(2)$] 

Assume that $n=6m+2$. 
\begin{enumerate}
\item [{\rm (a)}]
If $m \neq 11$, $\delta(C)=s(C)=3$. 
\item [{\rm (b)}]
If $m=11$, $\delta(C)=s(C)=3, 5, 6$ or $7$. 
\end{enumerate}

\item[$(3)$] 
Assume that $n=6m+4$. 
\begin{enumerate}
\item [{\rm (a)}]
If $m \in \{1,2,4,13\}$, 
$\delta(C)=s(C)=1$. 
\item [{\rm (b)}]
If $m\in \{3,5,6,7,8,10,11,12,15,16,17,18,20,21\}$, $\delta(C)=s(C)=1$ or $3$. 
\item [{\rm (c)}]
If $m=9$, $\delta(C)=s(C)=1,3$ or $4$. 
\item [{\rm (d)}]
If $m \in \{ 14,19 \}$, $\delta(C)=s(C)=1,3,4$ or $5$. 
\end{enumerate}

\end{enumerate}

\end{Thm}

Summarizing the above, 
we have the following theorem.
\begin{Thm}
Let $C$ be an extremal Type III or IV code. 
\begin{enumerate}
\item 

[An answer to the Problem \ref{problem:1}]

We have $s(C) \leq 7$. 

\item 

[An answer to the Problem \ref{problem:2}]

The case $\delta(C) < s(C)$ does not occur.

\end{enumerate}
\end{Thm}


This paper is organized as follows. 
In Section \ref{sec:pre}, 
we provide the definitions and some basic properties of 
self-dual codes and $t$-designs
as well as
review the concept of 
harmonic weight enumerators and some lemmas 
that are used in the proof of the main results. 
In Sections \ref{sec: proof of thm1.3} and \ref{sec: proof of thm1.4}, 
we provide the proofs of Theorems~\ref{thm:main upper bound III} and \ref{thm:main upper bound IV}, 
respectively.

All computer calculations were performed using 
Mathematica \cite{Mathematica}.

\section{Preliminaries}\label{sec:pre}

\subsection{Codes and support $t$-designs}

Let $\GF_q$ be a finite field of $q$ elements. 
A linear code $C$ of length $n$ is a linear subspace of $\FF_{q}^{n}$. 
For $q=3$, 
an inner product $({x},{y})$ on $\FF_q^n$ is expressed as
\[
(x,y)=\sum_{i=1}^nx_iy_i,
\]
where $x,y\in \FF_q^n$ with $x=(x_1,x_2,\ldots, x_n)$, and 
$y=(y_1,y_2,\ldots, y_n)$. 
The Hermitian inner product $({x},{y})$ on $\FF_4^n$ is expressed as
\[
(x,y)_H=\sum_{i=1}^nx_iy_i^2,
\]
where $x,y\in \FF_4^n$ with $x=(x_1,x_2,\ldots, x_n)$, and 
$y=(y_1,y_2,\ldots, y_n)$. 
The dual of a linear code $C$ is defined as follows: 
for $q=3$, 
\[
C^{\perp}=\{{y}\in \FF_{q}^{n}\ | \ ({x},{y}) =0\ \mbox{ for all }{x}\in C\},
\]
for $q=4$, 
\[
C^{\perp,H}=\{{y}\in \FF_{q}^{n}\ | \ ({x},{y})_H =0\ \mbox{ for all }{x}\in C\}.
\]
A linear code $C$ is called self-dual 
if $C=C^{\perp}$ for $q=3$, and 
if $C=C^{\perp,H}$ for $q=4$. 
For $x \in\FF_q^n$,
the weight $\wt(x)$ is the number of its nonzero components. 
The minimum distance of a code $C$ is 
$\min\{\wt( x)\mid  x \in C,  x \neq  0 \}$. 
A linear code of length $n$, dimension $k$, and 
minimum distance $d$ is called an $[n,k,d]$ code. 

Herein, we consider the following self-dual codes~\cite{CS}: 
\begin{tabbing}
Type III: A code is defined over $\FF_{3}^{n}$ with all weights divisible by $3$,\\
Type IV: A code is defined over $\FF_{4}^{n}$ with all weights divisible by $2$. \
\end{tabbing}

Let $C$ be a Type III or Type IV code of length $n$. 
Then we have the following bound on the minimum weight of 
$C$ \cite{MOSW1978,mallows-sloane}:
\begin{equation}\label{34}
\min(C)\leq 
\begin{cases}
3\left\lfloor\frac{n}{12}\right\rfloor+3
&\text{if $C$ is ternary,}\\
2\left\lfloor\frac{n}{6}\right\rfloor+2
&\text{if $C$ is quaternary.}
\end{cases}
\end{equation}
We say that $C$ meeting
the bound (\ref{34}) with equality is extremal.

A $t$-$(v,k,{\lambda})$ design (or $t$-design for short) is a pair 
$\mathcal{D}=(X,\mathcal{B})$, where $X$ is a set of points of 
cardinality $v$, and $\mathcal{B}$ a collection of $k$-element subsets
of $X$ called blocks, with the property that any $t$ points are 
contained in precisely $\lambda$ blocks.

The support of a nonzero vector $ x:=(x_{1}, \dots, x_{n})$, $x_{i} \in \GF_{q} = \{ 0,1, \dots, q-1 \}$ is 
the set of indices of its nonzero coordinates: ${\rm supp} ( x) = \{ i \mid x_{i} \neq 0 \}$\index{$supp (x)$}. 
The support design of a code of length $n$ for a nonzero weight $w$ is a design 
with $n$ points of coordinate indices; it blocks the supports of all codewords of weight $w$.
The following lemma can be observed easily. 
\begin{Lem}[{{\cite[Page 3, Proposition 1.4]{CL}}}]\label{lem: divisible}
Let $\lambda(S)$ be the number of blocks containing a set $S$
of $s$ points in a $t$-$(v,k,\lambda)$ design, where $0\leq s\leq t$. Therefore,
\[
\lambda(S)\binom{k-s}{t-s}
=
\lambda\binom{v-s}{t-s}. 
\]
In particular, the number of blocks is
\[\frac{v(v-1)\cdots(v-t+1)}{k(k-1)\cdots(k-t+1)}\lambda.\]
\end{Lem}




\subsection{Harmonic weight enumerators}\label{sec:Har}

In this section, we extend the harmonic weight enumerator method used by Bachoc \cite{Bachoc} and 
Bannai et al.~\cite{Bannai-Koike-Shinohara-Tagami}.
For convenience, we quote (from \cite{Bachoc,Delsarte})
the definitions and properties of discrete harmonic functions 
(for more information, the reader is referred to \cite{Bachoc,Delsarte}). 


Let $\Omega=\{1, 2,\ldots,n\}$ be a finite set (which will be the set of coordinates of the code), and 
let $X$ be the set of its subsets; for all $k= 0,1, \ldots, n$, $X_{k}$ is the set of its $k$-subsets.
We denote the free real vector spaces spanned by the elements of $X$ and $X_{k}$ by $\R X$, $\R X_k$, respectively. 
The element of $\R X_k$ is denoted by
$$f=\sum_{z\in X_k}f(z)z$$
and is identified with the real-valued function on $X_{k}$ expressed as 
$z \mapsto f(z)$. 

Such an element $f\in \R X_k$ can be extended to an element $\widetilde{f}\in \R X$ by setting, for all $u \in X$,
$$\widetilde{f}(u)=\sum_{z\in X_k, z\subset u}f(z).$$
If an element $g \in \R X$ is equal to some $\widetilde{f}$, for $f \in \R X_{k}$, we say that $g$ has a degree of $k$. 
The differentiation $\gamma$ is the operator defined by linearity from 
$$\gamma(z) =\sum_{y\in X_{k-1},y\subset z}y$$
for all $z\in X_k$ and for all $k=0,1, \ldots n$, and $\Harm_{k}$ is the kernel of $\gamma$, i.e.,
$$\Harm_k =\ker(\gamma|_{\R X_k}).$$

\begin{Thm}[{{\cite[Theorem 7]{Delsarte}}}]\label{thm:design}
A set $\mathcal{B} \subset X_{m}$ (where $m \leq n$) of blocks is a $t$-design 
if and only if $\sum_{b\in \mathcal{B}}\widetilde{f}(b)=0$ 
for all $f\in \Harm_k$, $1\leq k\leq t$. 
\end{Thm}
In \cite{Bachoc}, the harmonic weight enumerator associated with a linear code $C$ is defined as follows: 
\begin{Def}[{\cite[Definition 2.1]{Bachoc}},{\cite[Definition 4.1]{Bachoc2}}]
Let $C$ be a linear code of length $n$, and let $f\in\Harm_{k}$. 
The harmonic weight enumerator associated with $C$ and $f$ is
\[
W_{C,f}(x,y)=\sum_{c\in C}\widetilde{f}(c)x^{n-\wt(c)}y^{\wt(c)}.
\]
\end{Def}

Subsequently, the structure of these invariant rings is described as follows: 
\begin{Thm}[{\cite[Lemma 6.1 and 6.2]{Bachoc2}}]\label{thm:invariant}

\begin{enumerate}
\item [{\rm (1)}]
Let $C$ be a Type III
code of length $n$, and let 
$f \in \Harm_{k}$. 
Let $u\in\{0,1\}$ be such that $u\equiv k \pmod 2$ and 
$v\in\{0,1,2\}$ be such that $v\equiv -k \pmod 3$. 
Therefore, we have $W_{C,f}(x,y) =(xy)^{k} Z_{C,f} (x,y)$. 
Moreover, the polynomial $Z_{C,f} (x,y)$ is of degree $n-2k$ 
and is in $I_{G_3, \chi_{u,v}}$, 
where 
\begin{align*}
I_{G_3,\chi_{u,v}}=
\left\{
\begin{array}{ll}
\langle g_4,g_{12}\rangle &\mbox{ if }(u,v)=(0,0),\\
p_{4}\langle g_4,g_{12}\rangle &\mbox{ if }(u,v)=(0,1),\\
p_{4}^2\langle g_4,g_{12}\rangle &\mbox{ if }(u,v)=(0,2),\\
p_{6}\langle g_4,g_{12}\rangle &\mbox{ if }(u,v)=(1,0),\\
p_{4}p_6\langle g_4,g_{12}\rangle &\mbox{ if }(u,v)=(1,1),\\
p_{4}^2p_6\langle g_4,g_{12}\rangle &\mbox{ if }(u,v)=(1,2), 
\end{array}
\right. 
\end{align*}
and
\begin{align*}
\left\{
\begin{array}{l}
p_{4}=y(x^3-y^3), \\
p_{6}=x^6-20x^3y^3-8y^6, \\
g_{4}=x^4+8xy^3, \\
g_{12}=y^3(x^3-y^3)^3. 
\end{array} 
\right. 
\end{align*}

\item [{\rm (2)}]

Let $C$ be a Type IV
code of length $n$, and let 
$f \in \Harm_{k}$. 
Let $u\in\{0,1\}$ be such that $u\equiv k \pmod 2$ and 
$v\in\{0,1\}$ be such that $v\equiv k \pmod 2$. 
Therefore, we have $W_{C,f}(x,y) =(xy)^{k} Z_{C,f} (x,y)$. Moreover, the polynomial $Z_{C,f} (x,y)$ is of degree $n-2k$ 
and is in $I_{G_4, \chi_{u,v}}$, 
where 
\begin{align*}
I_{G_4,\chi_{u,v}}=
\left\{
\begin{array}{ll}
\langle h_2,h_{6}\rangle &\mbox{ if }(u,v)=(0,0),\\
q_{3}r_3\langle h_2,h_{6}\rangle &\mbox{ if }(u,v)=(1,1), 
\end{array}
\right. 
\end{align*}
and 
\begin{align*}
\left\{
\begin{array}{l}
h_{2}=x^2+3y^2, \\
h_{6}=y^2(x^2-y^2)^2, \\
q_{3}=y(x^2-y^2), \\
r_{3}=x^3-9xy^2. 
\end{array}
\right. 
\end{align*}

\end{enumerate}
\end{Thm}

We recall the slightly more general definition of the notion of a $T$-design for a subset $T$ of $\{ 1,2, \ldots, n \}$, as follows: 
a set $\mathcal{B}$ of blocks is called a $T$-design if and only if $\sum_{b\in \mathcal{B}}\tilde{f}(b)=0$ 
for all $f\in \Harm_k$ and for all $k \in T$. 
By Theorem \ref{thm:design}, a $t$-design is 
a $T= \{1, \ldots, t \}$-design.
Let $W_{C,f}=\sum_{i=0}^{n}c_f(i)x^{n-i}y^i$. 
Subsequently, $D_w$ is a $T$-design if and only if $c_f(w)=0$ 
for all $f\in \mbox{Harm}_j$ with $j\in T$. 

\begin{Thm}[
\cite{{{Calderbank-Delsarte}}}]\label{thm:{Calderbank-Delsarte}}
\begin{enumerate}
\item [{\rm (1)}]
Let $D_{w}$ be the support design of weight $w$ of an extremal Type III code of length $n$ $($$n \geq 12$$)$. 
\begin{itemize}
\item If $n \equiv 0 \pmod{12}$, then $D_{w}$ is a $\{1, 2, 3, 4, 5, 7\}$-design.
\item If $n \equiv 4 \pmod{12}$, then $D_{w}$ is a $\{1, 2, 3, 5\}$-design. 
\item If $n \equiv 8 \pmod{12}$, then $D_{w}$ is a $\{1, 3\}$-design.
\end{itemize}
\item [{\rm (2)}]
Let $D_{w}$ be the support design of weight $w$ of an extremal Type IV code of length $n$. 
\begin{itemize}
\item If $n \equiv 0 \pmod{6}$ $($$n \geq 18$$)$, then $D_{w}$ is a $\{1, 2, 3, 4, 5, 7\}$-design.
\item If $n \equiv 2 \pmod{6}$, then $D_{w}$ is a $\{1, 2, 3, 5\}$-design. 
\item If $n \equiv 4 \pmod{6}$, then $D_{w}$ is a $\{1, 3\}$-design.
\end{itemize}
\end{enumerate}
\end{Thm}


\subsection{Coefficients of harmonic weight enumerators of extremal Type III and IV codes}

As mentioned in Section \ref{sec:Har}, 
the support designs of a code $C$ are affected by whether the coefficients of $W_{C,f}(x,y)$ are zero.
Therefore, we performed an investigation and show the following lemmas, 
where the binomial coefficient is defined by 
\[
\binom{n}{k}=0
\] if $n<k$.

\begin{Lem}\label{lem:poly. zero 1}
Let $Q_{1}=(x^{4}+8xy^{3})(x^{3}-y^{3})^{\alpha}$.
If the coefficients of $x^{3\alpha+4-3i}y^{3i}$ in $Q_{1}$ are equal to $0$ 
for $0 \leq i \leq \alpha+1$, then $\alpha=9i-1$. 
\end{Lem}

\begin{proof}
We have 
\begin{align*}
Q_{1}&=(x^{4}+8xy^{3})(x^{3}-y^{3})^{\alpha} \\
     &=\sum_{i=0}^{\alpha+1} (-1)^{i} \left( \binom{\alpha}{i}-8\binom{\alpha}{i-1} \right)x^{3\alpha+4-3i}y^{3i}.
\end{align*}
If the coefficients of $x^{3\alpha+4-3i}y^{3i}$ in $Q_{1}$ are equal to $0$, i.e., 
\[
\binom{\alpha}{i}-8\binom{\alpha}{i-1}=0,
\]
we then have 
\begin{align*}
&\frac{\alpha !}{i !(\alpha-i)!}-8\frac{\alpha !}{(i-1) !(\alpha-i+1)!}=0 \\
\Leftrightarrow\ &\alpha-i+1-8i=0 \\
\Leftrightarrow\ &\alpha=9i-1.
\end{align*}

\end{proof}

\begin{Lem}\label{lem:poly. zero 2}
\begin{enumerate}
\item[{\rm (1)}]
Let $R_{1}=(x^{2}+3y^{2})(x^{2}-y^{2})^{\alpha}$.
If the coefficients of 
$x^{2\alpha+2-2i}y^{2i}$ in $R_{1}$ are equal to $0$ for $0 \leq i \leq \alpha+1$, 
then $\alpha=4i-1$.
\item[{\rm (2)}]
Let $R_{2}=(x^{3}-9xy^{2})(x^{2}-y^{2})^{\alpha}$.
If the coefficients of 
$x^{2\alpha+3-2i}y^{2i}$ in $R_{2}$ are not equal to $0$. 
\item[{\rm (3)}]
Let $R_{3}=(x^{2}+3y^{2})^{2}(x^{2}-y^{2})^{\alpha}$.
If the coefficients of 
$x^{2\alpha+4-2i}y^{2i}$ in $R_{3}$ are equal to $0$ for $0 \leq i \leq \alpha+2$, 
then $48\alpha+112$ is a square number.
\end{enumerate}
\end{Lem}
\begin{proof}
(1)
We have 
\begin{align*}
R_{1}&=(x^{2}+3y^{2})(x^{2}-y^{2})^{\alpha} \\
 &=\sum_{i=0}^{\alpha+1} (-1)^{i} \left( \binom{\alpha}{i}-3\binom{\alpha}{i-1} \right)x^{2\alpha+2-2i}y^{2i}.
\end{align*}
If the coefficients of $x^{2\alpha+2-2i}y^{2i}$ in $R_{1}$ are equal to $0$, i.e.,
\[
\binom{\alpha}{i}-3\binom{\alpha}{i-1}=0,
\]
we then have 
\begin{align*}
&\frac{\alpha !}{i !(\alpha-i)!}-3\frac{\alpha !}{(i-1) !(\alpha-i+1)!}=0 \\
\Leftrightarrow\ &\alpha-i+1-3i=0 \\
\Leftrightarrow\ &\alpha=4i-1.
\end{align*}

(2)
We have 
\begin{align*}
R_{2}&=(x^{3}-9xy^{2})(x^{2}-y^{2})^{\alpha} \\
 &=\sum_{i=0}^{\alpha+1} (-1)^{i} \left( \binom{\alpha}{i}+9\binom{\alpha}{i-1} \right)x^{2\alpha+3-2i}y^{2i}.
\end{align*}
If the coefficients of $x^{2\alpha+3-2i}y^{2i}$ in $R_{2}$ are equal to $0$, i.e., 
\[
\binom{\alpha}{i}+9\binom{\alpha}{i-1}=0, 
\]
we then have 
\begin{align*}
&\frac{\alpha !}{i !(\alpha-i)!}+9\frac{\alpha !}{(i-1) !(\alpha-i+1)!}=0 \\
\Leftrightarrow\ &\alpha-i+1+9i=0 \\
\Leftrightarrow\ &\alpha=-8i-1<0.
\end{align*}
Hence, the coefficients of $x^{2\alpha+3-2i}y^{2i}$ in $R_{2}$ are not equal to $0$.

(3)
We have 
\begin{align*}
R_{3}&=(x^{2}+3y^{2})^{2}(x^{2}-y^{2})^{\alpha} \\
 &=\sum_{i=0}^{\alpha+1} (-1)^{i} \left( \binom{\alpha}{i}-6\binom{\alpha}{i-1}+9\binom{\alpha}{i-2} \right)x^{2\alpha+4-2i}y^{2i}.
\end{align*}
If the coefficients of $x^{2\alpha+4-2i}y^{2i}$ in $R_{3}$ are equal to $0$, i.e., 
\[
\binom{\alpha}{i}-6\binom{\alpha}{i-1}+9\binom{\alpha}{i-2}=0,
\]
we then have 
\begin{align*}
&\frac{\alpha !}{i !(\alpha-i)!}-6\frac{\alpha !}{(i-1) !(\alpha-i+1)!}+9\frac{\alpha !}{(i-2) !(\alpha-i+2)!}=0 \\
\Leftrightarrow\ &(\alpha-i+2)(\alpha-i+1)-6i(\alpha-i+2)+9i(i-1)=0 \\
\Leftrightarrow\ &16i^{2}-(8\alpha+24)i+\alpha^{2}+3\alpha+2=0.
\end{align*}
We have 
\[
i=\frac{4\alpha+12 \pm \sqrt{48\alpha+112}}{16}.
\]
Because $i$ is an integer, $48\alpha+112$ is a square number.

\end{proof}

\section{Proof of Theorem \ref{thm:main upper bound III}}\label{sec: proof of thm1.3}
\subsection{Case for $n=12m$}\label{sec: case for n=12m}

In this section, we consider the case of extremal 
Type III $[12m,6m,3m+3]$ codes 
satisfying (\ref{boundIII}). 
Let $C$ be an extremal Type III $[12m,6m,3m+3]$ code and 
$D_{3m+3}^{12m}$ be the support (with duplicates omitted) design of the minimum weight of $C$.
By \cite[Theorem 2]{mallows-sloane}, 
the number of codewords of minimum nonzero weight of $C$ is equal to 
\[
2\binom{12m}{5}\binom{4m-2}{m-1}\left/\binom{3m+3}{5}.\right.
\]
Therefore, by the Assmus--Mattson theorem, $D_{3m+3}^{12m}$ is a $5$-design 
with parameters 
\[\left(12m,3m+3,\binom{4m-2}{m-1} \right).\] 

\begin{Prop}\label{prop:type3 12m not 8-design}

\begin{enumerate}
\item [{\rm (1)}]
If $t \geq 6$, then $D_{3m+3}^{12m}$ is a $7$-design and 
$m=15$. 
\item [{\rm (2)}]
$D_{3m+3}^{12m}$ is never an $8$-design.
\end{enumerate}

\end{Prop}

\begin{proof}
(1)
By Theorem~\ref{thm:{Calderbank-Delsarte}} (1), $D_{3m+3}^{12m}$ is a $7$-design if $t \geq 6$.
If $D_{3m+3}^{12m}$ is a $7$-design, 
then by Lemma \ref{lem: divisible},
\[\lambda_{6}=\frac{3m-2}{12m-5}\binom{4m-2}{m-1}\ {\rm and}\ 
\lambda_{7}=\frac{(3m-2)(3m-3)}{(12m-5)(12m-6)}\binom{4m-2}{m-1}\] 
are positive integers.
By computing 
$m$ satisfying (\ref{boundIII}), 
if $\lambda_{6}$ and $\lambda_{7}$ are positive integers, 
then we have 
$m=15$.

(2)
We have verified that 
\[\lambda_{8}=\frac{(3m-2)(3m-3)(3m-4)}{(12m-5)(12m-6)(12m-7)}\binom{4m-2}{m-1}
\]
is not a positive integer for 
$m =15$. 
Therefore, by Lemma \ref{lem: divisible},
$D_{3m+3}^{12m}$ is never an $8$-design.
\end{proof}

For $t=8$, we present the following proposition. 

\begin{Prop}\label{prop:type3 12m harm}
Let $D_{w}^{12m}$ be the support $t$-design of weight $w$ of 
an extremal Type III code of length $n=12m$. 
Therefore, all $D_{w}^{12m}$ are $8$-designs simultaneously, or none of the $D_{w}^{12m}$ is an $8$-design. 
\end{Prop}
\begin{proof}

Let us assume that $t=8$, and $C$ is an extremal Type III $[12m,6m,3m+3]$ code.
Therefore, by Theorem~\ref{thm:invariant} (1), we have $W_{C,f}(x,y) =c(f) (xy)^{8} Z_{C,f} (x,y)$, 
where $c(f)$ is a linear function from Harm$_{t}$ to $\R$, and $Z_{C,f} (x,y) \in I_{G_{3},\chi_{0,1}}$.
By Theorem \ref{thm:invariant} (1), 
$Z_{C,f} (x,y)$ can be written in the following form:
\[
Z_{C,f}(x,y) = p_{4} \sum_{i=0}^{m}a_{i}g_{4}^{3(m-i)-5} g_{12}^{i}.
\]
Because the minimum weight of $C$ is $3m+3$, 
we have $a_{i}=0$ for $i \neq m-2$. 
Therefore, $W_{C,f}(x,y)$ can be written in the following form: 
\begin{align*}
W_{C,f}(x,y) &=c(f) (xy)^{8} p_{4}g_{4} g_{12}^{m-2}   \\
& =c(f) (xy)^{8} y^{3m-5} (x^{4}+8xy^{3}) (x^{3}-y^{3})^{3m-5}. 
\end{align*}

By Lemma~\ref{lem:poly. zero 1}, 
the coefficients of $x^{9m-11-3i}y^{3i}$ in $(x^{4}+8xy^{3}) (x^{3}-y^{3})^{3m-5}$ are not equal to $0$ for $0 \leq i \leq 3m-4$
because $3m-5 \neq 9i-1$.
Therefore, all $D_{w}^{12m}$ are $8$-designs simultaneously, or none of the $D_{w}^{12m}$ is an $8$-design. 
\end{proof}

By Propositions \ref{prop:type3 12m not 8-design} and \ref{prop:type3 12m harm},
we obtained the following theorem.

\begin{Thm}\label{thm:main thm 12m}
\item[{\rm (1)}]
If $D_{w}^{12m}$ becomes a $7$-design for any $w$, 
then $m=15$. 
\item[{\rm (2)}] $D_{w}^{12m}$ is never an $8$-design for any $w$.
\end{Thm}

Hence, the proof of Theorem \ref{thm:main upper bound III} (1) is completed.


\subsection{Case for $12m+4$}\label{sec: case for n=12m+4}

In this section, we consider the case of extremal Type III $[12m+4,6m+2,3m+3]$ codes 
satisfying (\ref{boundIII}). 
Let $C$ be an extremal Type III $[12m+4,6m+2,3m+3]$ code 
and $D_{3m+3}^{12m+4}$ be the support (with duplicates omitted) design of the minimum weight of $C$.
By \cite[Theorem 2]{mallows-sloane}, 
the number of codewords of the minimum nonzero weight of $C$ is equal to 
\[
2(12m+4)(12m+3)(12m+2)\frac{(4m)!}{m!(3m+3)!}.
\]
Therefore, by the Assmus--Mattson theorem, $D_{3m+3}^{12m+4}$ is a $3$-design 
with parameters 
\[\left(12m+4,3m+3,\binom{4m}{m} \right).
\]

\begin{Prop}\label{prop:type3 12m+4 not 7-design}
Let $D_{3m+3}^{12m+4}$ be the support $t$-design of the minimum weight of an extremal Type III code of length $n=12m+4$. 

\begin{enumerate}
\item [{\rm (1)}]
If $t \geq 4$, then $D_{3m+3}^{12m+4}$ is a $5$-design and $m$ must be in the set \\
$\{11,21,25\}$. 
\item [{\rm (2)}]
$D_{3m+3}^{12m+4}$ is never a $6$-design.
\end{enumerate}
\end{Prop}

\begin{proof}
(1)
By Theorem~\ref{thm:{Calderbank-Delsarte}} (1), $D_{3m+3}^{12m+4}$ is a $5$-design if $t \geq 4$.
If $D_{3m+3}^{12m+4}$ is a $5$-design, 
then by Lemma \ref{lem: divisible},
\[
\lambda_{4}=\frac{3m}{12m+1}\binom{4m}{m} \mbox{ and } 
\lambda_{5}=\frac{3m(3m-1)}{(12m+1)12m}\binom{4m}{m}
\] are positive integers.
By computing 
$m$ satisfying (\ref{boundIII}), 
 if $\lambda_{4}$ and $\lambda_{5}$ are positive integers, 
then we have 
\[
m \in \{11,21,25\}.
\]

(2)If $D_{3m+3}^{12m+4}$ is a $6$-design, then
by Lemma \ref{lem: divisible},
\[
\lambda_{6}=\frac{3m(3m-1)(3m-2)}{(12m+1)12m(12m-1)}\binom{4m}{m} 
\]
is a positive integer. 
Then we do not obtain $m$ 
satisfying (\ref{boundIII}).

\end{proof}

For $t \geq 6$, we present the following proposition. 

\begin{Prop}\label{prop:type3 12m+4 harm}
Let $D_{w}^{12m+4}$ be the support $t$-design of weight $w$ of 
an extremal Type III code of length $n=12m+4$. 
\begin{enumerate}
\item[{\rm (1)}] All $D_{w}^{12m+4}$ are $6$-designs simultaneously, or none of the $D_{w}^{12m+4}$ is a $6$-design. 
\end{enumerate}
\end{Prop}
\begin{proof}
Let $C$ be an extremal Type III $[12m+4,6m+2,3m+3]$ code.

(1)
Let us assume that $t=6$.
Therefore, by Theorem~\ref{thm:invariant} (1), we have $W_{C,f}(x,y) =c(f) (xy)^{6} Z_{C,f} (x,y)$, 
where $c(f)$ is a linear function from Harm$_{t}$ to $\R$, and $Z_{C,f} (x,y) \in I_{G_{3},\chi_{0,0}}$.
By Theorem \ref{thm:invariant} (1), 
$Z_{C,f} (x,y)$ can be written in the following form:
\[
Z_{C,f}(x,y) =  \sum_{i=0}^{m}a_{i}g_{4}^{3(m-i)-2} g_{12}^{i}.
\]
Because the minimum weight of $C$ is $3m+3$, 
we have $a_{i}=0$ for $i \neq m-1$. 
Therefore, $W_{C,f}(x,y)$ can be written in the following form: 
\begin{align*}
W_{C,f}(x,y) &=c(f) (xy)^{6} g_{4} g_{12}^{m-1}  \\
& =c(f) (xy)^{6} y^{3m-3} (x^{4}+8xy^{3}) (x^{3}-y^{3})^{3m-3}. 
\end{align*}

By Lemma~\ref{lem:poly. zero 1}, 
the coefficients of $x^{9m-5-3i}y^{3i}$ in $(x^{4}+8xy^{3}) (x^{3}-y^{3})^{3m-3}$ are not equal to $0$ for $0 \leq i \leq 3m-2$
since $3m-3 \neq 9i-1$.
Therefore, all $D_{w}^{12m+4}$ are $6$-designs simultaneously, or none of the $D_{w}^{12m+4}$ is a $6$-design. 



\end{proof}

By Proposition \ref{prop:type3 12m+4 not 7-design} and \ref{prop:type3 12m+4 harm},
we obtained the following theorem.

\begin{Thm}\label{thm:main thm 12m+4}
Let $D_{w}^{12m+4}$ be the support $t$-design of weight $w$ of an extremal Type III code of length $n=12m+4$ 
satisfying (\ref{boundIII}). 
\begin{enumerate}
\item[{\rm (1)}]
If $D_{w}^{12m+4}$ becomes a $5$-design for any $w$, then $m$ must be in the set\\ 
$\{11,21,25\}$. 
\item[{\rm (2)}]
$D_{w}^{12m+4}$ is never a $6$-design for any $w$.
\end{enumerate}
\end{Thm}

Hence, the proof of Theorem \ref{thm:main upper bound III} (2) is completed.


\subsection{Case for $12m+8$}\label{sec: case for n=12m+8}

In this section, we consider the case of extremal Type III $[12m+8,6m+4,3m+3]$ codes 
satisfying (\ref{boundIII}). 
Let $C$ be an extremal Type III $[12m+8,6m+4,3m+3]$ code 
and $D_{3m+3}^{12m+8}$ be the support (with duplicates omitted) design of the minimum weight of $C$.
By \cite[Theorem 2]{mallows-sloane}, 
the number of codewords of the minimum nonzero weight of $C$ is equal to 
\[
6(12m+8)\frac{(4m+2)!}{m!(3m+3)!}.
\]
Therefore, by the Assmus--Mattson theorem, $D_{3m+3}^{12m+8}$ is a $1$-design 
with parameters 
\[
\left(12m+8,3m+3,3\binom{4m+2}{m} \right).
\]

\begin{Prop}\label{prop:type3 12m+8 not 4-design}
Let $D_{3m+3}^{12m+8}$ be the support $t$-design of the minimum weight of an extremal Type III code of length $n=12m+8$. 

\begin{enumerate}
\item [{\rm (1)}]
If $t \geq 2$, then $D_{3m+3}^{12m+8}$ is a $3$-design and $m$ must be $14$.
\item [{\rm (2)}]
$D_{3m+3}^{12m+8}$ is never a $4$-design.

\end{enumerate}

\end{Prop}

\begin{proof}
(1)
By Theorem~\ref{thm:{Calderbank-Delsarte}} (1), $D_{3m+3}^{12m+8}$ is a $3$-design if $t \geq 2$.
If $D_{3m+3}^{12m+8}$ is a $3$-design, 
then by Lemma \ref{lem: divisible},
\[\lambda_{2}=\frac{3m+2}{12m+7}3\binom{4m+2}{m}\ {\rm and}\ 
\lambda_{3}=\frac{(3m+2)(3m+1)}{(12m+7)(12m+6)}3\binom{4m+2}{m}
\]
 are positive integers.
By computing 
$m$ satisfying (\ref{boundIII}), 
if $\lambda_{2}$ and $\lambda_{3}$ are positive integers, 
then we have 
$m=14$.

(2)
We have verified that 
\[
\lambda_{4}=\frac{(3m+2)(3m+1)3m}{(12m+7)(12m+6)(12m+5)}3\binom{4m+2}{m}
\]
is not a  positive integer for 
$m=14$.
Therefore, by Lemma \ref{lem: divisible},
$D_{3m+3}^{12m+8}$ is never a $4$-design.
\end{proof}

Next, we present the following proposition.

\begin{Prop}\label{prop:type3 12m+8 harm}
Let $D_{w}^{12m+8}$ be the support $t$-design of weight $w$ of 
an extremal Type III code of length $n=12m+8$. 
If $m \not\equiv 0 \pmod 3$, 
all $D_{w}^{12m+8}$ are $4$-designs simultaneously, or none of the $D_{w}^{12m+8}$ is a $4$-design. 

\end{Prop}

\begin{proof}
Let $C$ be an extremal Type III $[12m+8,6m+4,3m+3]$ code.
Let us assume that $t=4$.
Therefore, by Theorem~\ref{thm:invariant} (1), we have $W_{C,f}(x,y) =c(f) (xy)^{4} Z_{C,f} (x,y)$, 
where $c(f)$ is a linear function from Harm$_{t}$ to $\R$, and $Z_{C,f} (x,y) \in I_{G_{3},\chi_{0,2}}$.
By Theorem \ref{thm:invariant} (1), 
$Z_{C,f} (x,y)$ can be written in the following form:
\[
Z_{C,f}(x,y) =  p_{4}^{2}\sum_{i=0}^{m}a_{i}g_{4}^{3(m-i)-2} g_{12}^{i}.
\]
Because the minimum weight of $C$ is $3m+3$, 
we have $a_{i}=0$ for $i \neq m-1$. 
Therefore, $W_{C,f}(x,y)$ can be written in the following form: 
\begin{align*}
W_{C,f}(x,y) &=c(f) (xy)^{4} p_{4}^{2}g_{4} g_{12}^{m-1}  \\
& =c(f) (xy)^{4} y^{3m-1} (x^{4}+8xy^{3}) (x^{3}-y^{3})^{3m-1}. 
\end{align*}

By Lemma~\ref{lem:poly. zero 1}, 
if $m \not\equiv 0 \pmod 3$, 
then the coefficients of $x^{9m+1-3i}y^{3i}$ in $(x^{4}+8xy^{3}) (x^{3}-y^{3})^{3m-1}$ are not equal to $0$ for $0 \leq i \leq 3m$
because $3m-3 \neq 9i-1$.
Therefore, all $D_{w}^{12m+8}$ are $4$-designs simultaneously, or none of the $D_{w}^{12m+8}$ is a $4$-design for $m \not\equiv 0 \pmod 3$. 


\end{proof}

By Propositions \ref{prop:type3 12m+8 not 4-design} and \ref{prop:type3 12m+8 harm},
we obtained the following theorem.

\begin{Thm}\label{thm:main thm 12m+8}
Let $D_{w}^{12m+8}$ be the support $t$-design of weight $w$ of an extremal Type III code of length $n=12m+8$ 
satisfying (\ref{boundIII}). 
\begin{enumerate}
\item[{\rm (1)}]
If $D_{w}^{12m+8}$ becomes a $3$-design for any $w$, 
then $m=14$. 
\item[{\rm (2)}]
In the case where 
$m =14 $, 
$D_{w}^{12m+8}$ is a $1$ or $3$-design for any $w$.
\item[{\rm (3)}]
$D_{w}^{12m+8}$ is never a $4$-design for any $w$.
\end{enumerate}
\end{Thm}

Hence, the proof of Theorem \ref{thm:main upper bound III} (3) is completed. 


\section{Proof of Theorem \ref{thm:main upper bound IV}}\label{sec: proof of thm1.4}
\subsection{Case for $n=6m$}\label{sec: case for n=6m}

In this section, we consider the case of extremal Type IV $[6m,3m,2m+2]$ codes ($3 \leq m \leq 16$). 
Let $C$ be an extremal Type IV $[6m,3m,2m+2]$ code 
and $D_{2m+2}^{6m}$ be the support (with duplicates omitted) design of the minimum weight of $C$.
By \cite[Theorem 18]{MOSW1978}, 
$D_{2m+2}^{6m}$ is a 
\[
5\mbox{-}\left(6m,2m+2,\binom{3m-3}{m-2} \right)\mbox{ design}.
\]

\begin{Prop}\label{prop:type4 6m not 8-design}
Let $D_{2m+2}^{6m}$ be the support $t$-design of the minimum weight of an extremal Type IV code of length $n=6m$. 

\begin{enumerate}
\item [{\rm (1)}]
If $t \geq 6$, then $D_{2m+2}^{6m}$ is a $7$-design and $m$ must be in the set 
$\{10,15 \}$. 

\item [{\rm (2)}]
$D_{2m+2}^{6m}$ is never an $8$-design.

\end{enumerate}
\end{Prop}

\begin{proof}
(1)
If $D_{2m+2}^{6m}$ is a $7$-design, then
by Lemma \ref{lem: divisible},
\[
\lambda_{6}=\frac{2m-3}{6m-5}\binom{3m-3}{m-2} \mbox{ and } 
\lambda_{7}=\frac{(2m-3)(2m-4)}{(6m-5)(6m-6)}\binom{3m-3}{m-2}
\]
 are positive integers.
By computing $m \leq 16$, if $\lambda_{6}$ and $\lambda_{7}$ are positive integers, 
then we have $m \in \{10,15 \}$.

(2)
We have verified that 
\[
\lambda_{8}=\frac{(2m-3)(2m-4)(2m-5)}{(6m-5)(6m-6)(6m-7)}\binom{3m-3}{m-2}
\]
is not a positive integer for $m \in \{10,15 \}$.
Therefore, by Lemma \ref{lem: divisible},
$D_{2m+2}^{6m}$ is never an $8$-design.
\end{proof}

For $t \geq 8$, we present the following proposition. 

\begin{Prop}\label{prop:type4 6m harm}
Let $D_{w}^{6m}$ be the support $t$-design of weight $w$ of 
an extremal Type IV code of length $n=6m$. 
Therefore, all $D_{w}^{6m}$ are $8$-designs simultaneously, or none of the $D_{w}^{6m}$ is an $8$-design. 
\end{Prop}
\begin{proof}
Let $C$ be an extremal Type IV $[6m,3m,2m+2]$ code.
Let us assume that $t=8$.
Therefore, by Theorem~\ref{thm:invariant} (2), we have $W_{C,f}(x,y) =c(f) (xy)^{8} Z_{C,f} (x,y)$, 
where $c(f)$ is a linear function from Harm$_{t}$ to $\R$, and $Z_{C,f} (x,y) \in I_{G_{4},\chi_{0,0}}$.
By Theorem \ref{thm:invariant} (2), 
$Z_{C,f} (x,y)$ can be written in the following form:
\[
Z_{C,f}(x,y) = \sum_{i=0}^{m}a_{i}h_{2}^{3(m-i)-8} h_{6}^{i}.
\]
Because the minimum weight of $C$ is $2m+2$, 
we have $a_{i}=0$ for $i \neq m-3$. 
Therefore, $W_{C,f}(x,y)$ can be written in the following form: 
\begin{align*}
W_{C,f}(x,y) &=c(f) (xy)^{8} h_{2} h_{6}^{m-3}   \\
& =c(f) (xy)^{8} y^{2m-6} (x^{2}+3y^{2}) (x^{2}-y^{2})^{2m-6}. 
\end{align*}

By Lemma~\ref{lem:poly. zero 2} (1), 
the coefficients of $x^{4m-10-2i}y^{2i}$ in $(x^{2}+3y^{2}) (x^{2}-y^{2})^{2m-6}$ 
are not equal to $0$ for $0 \leq i \leq 2m-5$
because $2m-6 \neq 4i-1$.
Therefore, all $D_{w}^{6m}$ are $8$-designs simultaneously, or none of the $D_{w}^{6m}$ is an $8$-design. 

\end{proof}

By Propositions \ref{prop:type4 6m not 8-design} and \ref{prop:type4 6m harm},
we obtained the following theorem.

\begin{Thm}\label{thm:main thm 6m}
\item[{\rm (1)}]
If $D_{w}^{6m}$ becomes a $7$-design for any $w$, then $m$ must be in the set $\{10,15 \}$. 
\item[{\rm (2)}]
 $D_{w}^{6m}$ is never an $8$-design for any $w$.
\end{Thm}

Hence, the proof of Theorem \ref{thm:main upper bound IV} (1) is completed.

\subsection{Case for $n=6m+2$}\label{sec: case for n=6m+2}

In this section, we consider the case of extremal Type IV $[6m+2,3m+1,2m+2]$ codes ($m \leq 19$). 
Let $C$ be an extremal Type IV $[6m+2,3m+1,2m+2]$ code 
and $D_{2m+2}^{6m+2}$ be the support (with duplicates omitted) design of the minimum weight of $C$.
By \cite[Theorem 14]{MOSW1978}, 
the number of codewords of the minimum nonzero weight of $C$ is equal to 
\[
\frac{3(6m+1)}{m+1}\binom{3m+1}{m}.
\]
Therefore, by the Assmus--Mattson theorem, $D_{2m+2}^{6m+2}$ is a $3$-design 
with parameters \[
\left(6m+2,2m+2,\frac{1}{3}\binom{3m}{m} \right). 
\]

\begin{Prop}\label{prop:type4 6m+2 not 8-design}
Let $D_{2m+2}^{6m+2}$ be the support $t$-design of the minimum weight of an extremal Type IV code of length $n=6m+2$. 

\begin{enumerate}
\item [{\rm (1)}]
If $t \geq 4$, then $D_{2m+2}^{6m+2}$ is a $5$-design and $m$ must be $11$. 
\item [{\rm (2)}]
$D_{2m+2}^{6m+2}$ is never an $8$-design.
\end{enumerate}
\end{Prop}

\begin{proof}
(1)
By Theorem~\ref{thm:{Calderbank-Delsarte}} (2), $D_{2m+2}^{6m+2}$ is a $5$-design if $t \geq 4$.
If $D_{2m+2}^{6m+2}$ is a $5$-design, then
by Lemma \ref{lem: divisible},
\[
\lambda_{4}=\frac{2m-1}{6m-1}\frac{1}{3}\binom{3m}{m} \mbox{ and } 
\lambda_{5}=\frac{(2m-1)(2m-2)}{(6m-1)(6m-2)}\frac{1}{3}\binom{3m}{m}
\]
are positive integers.
By computing $m \leq 19$, if $\lambda_{4}$ and $\lambda_{5}$ are positive integers, 
then we have $m=11$.

(2)
For $m=11$, 
we have verified that 
\begin{align*}
\lambda_{6}&=\frac{(2m-1)(2m-2)(2m-3)}{(6m-1)(6m-2)(6m-3)}\frac{1}{3}\binom{3m}{m},\mbox{ and }\\
\lambda_{7}&=\frac{(2m-1)(2m-2)(2m-3)(2m-4)}{(6m-1)(6m-2)(6m-3)(6m-4)}\frac{1}{3}\binom{3m}{m}
\end{align*}
are positive integers, and
\[
\lambda_{8}=\frac{(2m-1)(2m-2)(2m-3)(2m-4)(2m-5)}{(6m-1)(6m-2)(6m-3)(6m-4)(6m-5)}\frac{1}{3}\binom{3m}{m}
\]
is not a positive integer.
Therefore, by Lemma \ref{lem: divisible},
$D_{2m+2}^{6m+2}$ is never an $8$-design.
\end{proof}

For $t \geq 6$, we present the following proposition. 

\begin{Prop}\label{prop:type4 6m+2 harm}
Let $D_{w}^{6m+2}$ be the support $t$-design of weight $w$ of 
an extremal Type IV code of length $n=6m+2$. 
\begin{enumerate}
\item[{\rm (1)}] All $D_{w}^{6m+2}$ are $6$-designs simultaneously, or none of the $D_{w}^{6m+2}$ is a $6$-design. 
\item[{\rm (2)}] All $D_{w}^{6m+2}$ are $7$-designs simultaneously, or none of the $D_{w}^{6m+2}$ is a $7$-design. 
\item[{\rm (3)}] For the case $m=11$. 
All $D_{w}^{68}$ are $8$-designs simultaneously, or none of the $D_{w}^{68}$ is an $8$-design.
\end{enumerate}
\end{Prop}
\begin{proof}
Let $C$ be an extremal Type IV $[6m+2,3m+1,2m+2]$ code.

(1)
Let us assume that $t=6$. 
Therefore, by the Theorem~\ref{thm:invariant} (2), we have $W_{C,f}(x,y) =c(f) (xy)^{6} Z_{C,f} (x,y)$, 
where $c(f)$ is a linear function from Harm$_{t}$ to $\R$, and $Z_{C,f} (x,y) \in I_{G_{4},\chi_{0,0}}$.
By Theorem \ref{thm:invariant} (2), 
$Z_{C,f} (x,y)$ can be written in the following form:
\[
Z_{C,f}(x,y) = \sum_{i=0}^{m}a_{i}h_{2}^{3(m-i)-5} h_{6}^{i}.
\]
Because the minimum weight of $C$ is $2m+2$, 
we have $a_{i}=0$ for $i \neq m-2$. 
Therefore, $W_{C,f}(x,y)$ can be written in the following form: 
\begin{align*}
W_{C,f}(x,y) &=c(f) (xy)^{6} h_{2} h_{6}^{m-2}   \\
& =c(f) (xy)^{6} y^{2m-4} (x^{2}+3y^{2}) (x^{2}-y^{2})^{2m-4}. 
\end{align*}

By Lemma~\ref{lem:poly. zero 2} (1), 
the coefficients of $x^{4m-6-2i}y^{2i}$ in $(x^{2}+3y^{2}) (x^{2}-y^{2})^{2m-4}$ 
are not equal to $0$ for $0 \leq i \leq 2m-3$
because $2m-4 \neq 4i-1$.
Therefore, all $D_{w}^{6m+2}$ are $6$-designs simultaneously, or none of the $D_{w}^{6m+2}$ is a $6$-design. 

(2)
Let us assume that $t=7$. 
Therefore, by Theorem~\ref{thm:invariant} (2), we have $W_{C,f}(x,y) =c(f) (xy)^{7} Z_{C,f} (x,y)$, 
where $c(f)$ is a linear function from Harm$_{t}$ to $\R$, and $Z_{C,f} (x,y) \in I_{G_{4},\chi_{1,1}}$.
By Theorem \ref{thm:invariant} (2), 
$Z_{C,f} (x,y)$ can be written in the following form:
\[
Z_{C,f}(x,y) = q_{3}r_{3}\sum_{i=0}^{m}a_{i}h_{2}^{3(m-i)-9} h_{6}^{i}.
\]
Because the minimum weight of $C$ is $2m+2$, 
we have $a_{i}=0$ for $i \neq m-3$. 
Therefore, $W_{C,f}(x,y)$ can be written in the following form: 
\begin{align*}
W_{C,f}(x,y) &=c(f) (xy)^{7} q_{3}r_{3} h_{6}^{m-3}   \\
& =c(f) (xy)^{7} y^{2m-5} (x^{3}-9xy^{2}) (x^{2}-y^{2})^{2m-5}. 
\end{align*}

By Lemma~\ref{lem:poly. zero 2} (2), 
the coefficients of $x^{4m-7-2i}y^{2i}$ in $(x^{3}-9xy^{2}) (x^{2}-y^{2})^{2m-5}$ 
are not equal to $0$ for $0 \leq i \leq 2m-4$.
Therefore, all $D_{w}^{6m+2}$ are $7$-designs simultaneously, or none of the $D_{w}^{6m+2}$ is a $7$-design. 

(3)
Let us assume that $t=8$. 
Therefore, by Theorem~\ref{thm:invariant} (2), we have $W_{C,f}(x,y) =c(f) (xy)^{8} Z_{C,f} (x,y)$, 
where $c(f)$ is a linear function from Harm$_{t}$ to $\R$, and $Z_{C,f} (x,y) \in I_{G_{4},\chi_{0,0}}$.
By Theorem \ref{thm:invariant} (2), 
$Z_{C,f} (x,y)$ can be written in the following form:
\[
Z_{C,f}(x,y) = \sum_{i=0}^{m}a_{i}h_{2}^{3(m-i)-7} h_{6}^{i}.
\]
Because the minimum weight of $C$ is $2m+2$, 
we have $a_{i}=0$ for $i \neq m-3$. 
Therefore, $W_{C,f}(x,y)$ can be written in the following form: 
\begin{align*}
W_{C,f}(x,y) &=c(f) (xy)^{8} h_{2}^{2} h_{6}^{m-3}   \\
& =c(f) (xy)^{8} y^{2m-6} (x^{2}+3y^{2})^{2} (x^{2}-y^{2})^{2m-6}. 
\end{align*}

By Lemma~\ref{lem:poly. zero 2} (3), 
if the coefficients of $x^{4m-8-2i}y^{2i}$ in $(x^{2}+3y^{2})^{2} (x^{2}-y^{2})^{2m-6}$ 
are equal to $0$, then $48(2m-6)+112$ is a square number.
In the case where $m=11$, $48(2m-6)+112=880$ is not a square number.
Therefore, all $D_{w}^{68}$ are $8$-designs simultaneously, or none of the $D_{w}^{68}$ is an $8$-design. 

\end{proof}

By Propositions \ref{prop:type4 6m+2 not 8-design} and \ref{prop:type4 6m+2 harm},
we obtained the following theorem.

\begin{Thm}\label{thm:main thm 6m+2}
Let $D_{w}^{6m+2}$ be the support $t$-design of weight $w$ of an extremal Type IV code of length $n=6m+2$ $($$m \leq 19$$)$. 
\begin{enumerate}
\item[{\rm (1)}]
If $D_{w}^{6m+2}$ becomes a $5$-design for any $w$, then $m=11$.
\item[{\rm (2)}]
In the case where $m=11$, $D_{w}^{68}$ is a $3$-,$5$-,$6$-, or $7$-design for any $w$. 
\item[{\rm (3)}]
$D_{w}^{6m+2}$ is never an $8$-design for any $w$.
\end{enumerate}
\end{Thm}

Hence, the proof of Theorem \ref{thm:main upper bound IV} (2) is completed.

\subsection{Case for $n=6m+4$}\label{sec: case for n=6m+4}

In this section, we consider the case of extremal Type IV $[6m+4,3m+2,2m+2]$ codes ($m \leq 21$). 
Let $C$ be an extremal Type IV $[6m+4,3m+2,2m+2]$ code 
and $D_{2m+2}^{6m+4}$ be the support (with duplicates omitted) design of the minimum weight of $C$.
By \cite[Theorem 14]{MOSW1978}, 
the number of codewords of the minimum nonzero weight of $C$ is equal to 
\[
3\binom{3m+2}{m+1}.
\]
Therefore, by the Assmus--Mattson theorem, $D_{2m+2}^{6m+4}$ is a $1$-design 
with parameters 
\[
\left(6m+4,2m+2,\binom{3m+1}{m} \right). 
\]

\begin{Prop}\label{prop:type4 6m+4 not 6-design}
Let $D_{2m+2}^{6m+4}$ be the support $t$-design of the minimum weight of an extremal Type IV code of length $n=6m+4$. 
\begin{enumerate}
\item 
[{\rm (1)}]
If $D_{2m+2}^{6m+4}$ is a $3$-design, then $m$ must be in the set \\
$\{3,5,6,7,8,9,10,11,12,14,15,16,17,18,19,20,21\}$. 
\item
[{\rm (2)}]
If $D_{2m+2}^{6m+4}$ is a $4$-design, then $m$ must be in the set $\{ 9,14,19 \}$. 
\item 
[{\rm (3)}]
If $D_{2m+2}^{6m+4}$ is a $5$-design, then $m$ must be in the set $\{ 14,19 \}$. 
\item 
[{\rm (4)}]
$D_{2m+2}^{6m+4}$ is never a $6$-design.
\end{enumerate}
\end{Prop}

\begin{proof}
(1)
By Theorem~\ref{thm:{Calderbank-Delsarte}} (2), $D_{2m+2}^{6m+4}$ is a $3$-design if $t \geq 2$.
If $m \in \{1,2,4,13\}$, 
then $D_{2m+2}^{6m+4}$ is not a $3$-design. 

For $m\in \{1,4,13\}$. 
we have verified that both 
\[
\lambda_{2}=\frac{2m+1}{6m+3}\binom{3m+1}{m} \mbox{ or } 
\lambda_{3}=\frac{(2m+1)2m}{(6m+3)(6m+2)}\binom{3m+1}{m}
\]
are not positive integers. 

For $m=2$, it is known that no $3$-$(16,6,2)$ design exists \cite{Brouwer1977}.

(2)
If $D_{2m+2}^{6m+4}$ is a $4$-design, then
\[
\lambda_{4}=\frac{(2m+1)2m(2m-1)}{(6m+3)(6m+2)(6m+1)}\binom{3m+1}{m}
\]
is a positive integer.
By computing $m \leq 21$, if $\lambda_{4}$ is a positive integer, 
then we have $m \in \{ 9,14,19 \}$. 

(3)
For $m \in \{ 9,14,19 \}$, if 
\[
\lambda_{5}=\frac{(2m+1)2m(2m-1)(2m-2)}{(6m+3)(6m+2)(6m+1)6m}\binom{3m+1}{m}
\]
is a positive integer, then $m \in \{ 14,19 \}$. 

(4)
We have verified that 
\[\lambda_{6}=\frac{(2m+1)2m(2m-1)(2m-2)(2m-3)}{(6m+3)(6m+2)(6m+1)6m(6m-1)}\binom{3m+1}{m}
\] 
is not a positive integer for $m \in \{ 14,19 \}$.
\end{proof}

For $t \geq 4$, we present the following proposition. 

\begin{Prop}\label{prop:type4 6m+4 harm}
Let $D_{w}^{6m+4}$ be the support $t$-design of weight $w$ of 
an extremal Type IV code of length $n=6m+4$. 
\begin{enumerate}
\item[{\rm (1)}] All $D_{w}^{6m+4}$ are $4$-designs simultaneously, or none of the $D_{w}^{6m+4}$ is a $4$-design. 
\item[{\rm (2)}] All $D_{w}^{6m+4}$ are $5$-designs simultaneously, or none of the $D_{w}^{6m+4}$ is a $5$-design. 
\item[{\rm (3)}] If $m\in \{ 14,19 \}$, all $D_{w}^{6m+4}$ are $6$-designs simultaneously, 
or none of the $D_{w}^{6m+4}$ is a $6$-design.
\end{enumerate}
\end{Prop}
\begin{proof}
Let $C$ be an extremal Type IV $[6m+4,3m+2,2m+2]$ code.

(1)
Let us assume that $t=4$. 
Therefore, by Theorem~\ref{thm:invariant} (2), we have $W_{C,f}(x,y) =c(f) (xy)^{4} Z_{C,f} (x,y)$, 
where $c(f)$ is a linear function from Harm$_{t}$ to $\R$, and $Z_{C,f} (x,y) \in I_{G_{4},\chi_{0,0}}$.
By Theorem \ref{thm:invariant} (2), 
$Z_{C,f} (x,y)$ can be written in the following form:
\[
Z_{C,f}(x,y) = \sum_{i=0}^{m}a_{i}h_{2}^{3(m-i)-2} h_{6}^{i}.
\]
Because the minimum weight of $C$ is $2m+2$, 
we have $a_{i}=0$ for $i \neq m-1$. 
Therefore, $W_{C,f}(x,y)$ can be written in the following form: 
\begin{align*}
W_{C,f}(x,y) &=c(f) (xy)^{4} h_{2} h_{6}^{m-1}   \\
& =c(f) (xy)^{4} y^{2m-2} (x^{2}+3y^{2}) (x^{2}-y^{2})^{2m-2}. 
\end{align*}

By Lemma~\ref{lem:poly. zero 2} (1), 
the coefficients of $x^{4m-2-2i}y^{2i}$ in $(x^{2}+3y^{2}) (x^{2}-y^{2})^{2m-2}$ 
are not equal to $0$ for $0 \leq i \leq 2m-1$
because $2m-2 \neq 4i-1$.
Therefore, all $D_{w}^{6m+4}$ are $4$-designs simultaneously, or none of the $D_{w}^{6m+4}$ is a $4$-design. 

(2)
Let us assume that $t=5$. 
Therefore, by Theorem~\ref{thm:invariant} (2), we have $W_{C,f}(x,y) =c(f) (xy)^{5} Z_{C,f} (x,y)$, 
where $c(f)$ is a linear function from Harm$_{t}$ to $\R$, and $Z_{C,f} (x,y) \in I_{G_{4},\chi_{1,1}}$.
By Theorem \ref{thm:invariant} (2), 
$Z_{C,f} (x,y)$ can be written in the following form:
\[
Z_{C,f}(x,y) = q_{3}r_{3}\sum_{i=0}^{m}a_{i}h_{2}^{3(m-i)-6} h_{6}^{i}.
\]
Because the minimum weight of $C$ is $2m+2$, 
we have $a_{i}=0$ for $i \neq m-2$. 
Therefore, $W_{C,f}(x,y)$ can be written in the following form: 
\begin{align*}
W_{C,f}(x,y) &=c(f) (xy)^{5} q_{3}r_{3} h_{6}^{m-2}   \\
& =c(f) (xy)^{5} y^{2m-3} (x^{3}-9xy^{2}) (x^{2}-y^{2})^{2m-3}. 
\end{align*}

By Lemma~\ref{lem:poly. zero 2} (2), 
the coefficients of $x^{4m-3-2i}y^{2i}$ in $(x^{3}-9xy^{2}) (x^{2}-y^{2})^{2m-3}$ 
are not equal to $0$ for $0 \leq i \leq 2m-2$.
Therefore, all $D_{w}^{6m+4}$ are $5$-designs simultaneously, or none of the $D_{w}^{6m+4}$ is a $5$-design. 

(3)
Let us assume that $t=6$. 
Therefore, by Theorem~\ref{thm:invariant} (2), we have $W_{C,f}(x,y) =c(f) (xy)^{6} Z_{C,f} (x,y)$, 
where $c(f)$ is a linear function from Harm$_{t}$ to $\R$, and $Z_{C,f} (x,y) \in I_{G_{4},\chi_{0,0}}$.
By Theorem \ref{thm:invariant} (2), 
$Z_{C,f} (x,y)$ can be written in the following form:
\[
Z_{C,f}(x,y) = \sum_{i=0}^{m}a_{i}h_{2}^{3(m-i)-4} h_{6}^{i}.
\]
Because the minimum weight of $C$ is $2m+2$, 
we have $a_{i}=0$ for $i \neq m-2$. 
Therefore, $W_{C,f}(x,y)$ can be written in the following form: 
\begin{align*}
W_{C,f}(x,y) &=c(f) (xy)^{6} h_{2}^{2} h_{6}^{m-2}   \\
& =c(f) (xy)^{6} y^{2m-4} (x^{2}+3y^{2})^{2} (x^{2}-y^{2})^{2m-4}. 
\end{align*}

By Lemma~\ref{lem:poly. zero 2} (3), 
if the coefficients of $x^{4m-4-2i}y^{2i}$ in $(x^{2}+3y^{2})^{2} (x^{2}-y^{2})^{2m-4}$ 
are equal to $0$, then $48(2m-4)+112$ is a square number.
If $m=14$ or $19$, then $48(2m-4)+112=1264$ or $48(2m-4)+112=1744$ is not a square number.
Therefore, if $m\in \{ 14,19 \}$, 
all $D_{w}^{6m+4}$ are $6$-designs simultaneously, or none of the $D_{w}^{6m+4}$ is a $6$-design. 

\end{proof}

By Propositions \ref{prop:type4 6m+4 not 6-design} and \ref{prop:type4 6m+4 harm},
we obtained the following theorem.

\begin{Thm}\label{thm:main thm 6m+4}
Let $D_{w}^{6m+4}$ be the support $t$-design of weight $w$ of an extremal Type IV code of length $n=6m+4$ $($$m \leq 21$$)$. 
\begin{enumerate}
\item[{\rm (1)}]
If $D_{w}^{6m+4}$ becomes a $3$-design for any $w$, then $m$ must be in the set \\
$\{3,5,6,7,8,9,10,11,12,14,15,16,17,18,19,20,21\}$.
\item[{\rm (2)}]
If $D_{w}^{6m+4}$ becomes a $4$-design for any $w$, then $m$ must be in the set $\{ 9,14,19 \}$.
\item[{\rm (3)}]
In the case where $m=9$, $D_{w}^{58}$ is a $1,3$ or $4$-design for any $w$. 
If $m \in \{ 14,19 \}$, $D_{w}^{6m+4}$ is a $1$-,$3$-,$4$- or $5$-design for any $w$.
\item[{\rm (4)}]
$D_{w}^{6m+4}$ is never a $6$-design for any $w$.
\end{enumerate}
\end{Thm}

Hence, the proof of Theorem \ref{thm:main upper bound IV} (3) is completed.

\section*{Acknowledgments}

The authors would also like to thank the anonymous reviewers for their
beneficial comments on an earlier version of the manuscript. 
The first named author is supported by JSPS KAKENHI (22K03277).


\end{document}